\documentclass[10pt]{article}
\usepackage{latexsym,amsfonts,amsmath,amsthm,amssymb, epsfig}
\usepackage{color}
\usepackage{hyperref}
\usepackage{todonotes}
\usepackage{caption}

\usetikzlibrary{arrows.meta}
\usetikzlibrary{shapes.geometric}
\usepackage{enumerate}

\usepackage{trimclip}
\newlength{\trianglewidth}
\newlength{\pluswidth}
\usepackage[vcentering,dvips]{geometry}
\newtheorem{thm}{Theorem}
\newtheorem{lem}[thm]{Lemma}

\theoremstyle{remark}
\newtheorem{rem}{Remark}

\theoremstyle{definition}

\newcommand{\norm}[1]{\left\|{#1}\right\|}
\newcommand{\abs}[1]{\vert{#1}\rvert}
\newcommand{\relu}{\mathrm{ReLU}}

\newcommand{\real}{\mathbb{R}}
\newcommand{\R}{\mathbb{R}}
\newcommand{\Q}{\mathbb{Q}}

\newcommand{\vol}{\mathcal{L}^n}
\newcommand{\aeq}{\overset{\Delta}{=}}
\newcommand{\GL}{\mathrm{GL}}

\newcommand{\be}{\begin{equation}}
\newcommand{\ee}{\end{equation}}

\newlength{\fixboxwidth}
\begin{document}

\title{Improved  universal approximation with neural networks  studied via affine-invariant subspaces of \(L_2(\R^n)\)}
\author{S. Probst\footnote{Friedrich-Alexander Universit\"at Erlangen, Mathematics Department, Cauerstr. 11, 91058 Erlangen, Germany. Email: \href{mailto:samuel.probst@fau.de}{samuel.probst@fau.de}} \quad \  and \quad  C. Schneider\footnote{Friedrich-Alexander Universit\"at Erlangen, Applied Mathematics III, Cauerstr. 11, 91058 Erlangen, Germany. Email: \href{mailto:schneider@math.fau.de}{schneider@math.fau.de}}}
\maketitle

\begin{abstract}

{\small We show that there are no non-trivial closed subspaces of \(L_2(\R^n)\) that are invariant under invertible affine transformations\footnote{See the introduction for a precise definition.}. We apply this result to neural networks showing that any nonzero $L_2(\real)$ function is an adequate activation function in a one hidden layer neural network in order to  approximate every  function in $L_2(\real)$ with any desired accuracy. This generalizes the universal approximation properties of neural networks in $L_2(\real)$ related to Wiener's Tauberian Theorems. Our results extend to the spaces $L_p(\real)$ with  $p>1$. \\
\noindent{\em Key Words:} Neural networks, activation functions, Wiener's Tauberian Theorems, Lebesgue spaces, affine invariant subspaces, Universal approximation theorem.\\
{\em MSC2020 Math Subject Classifications:} Primary: 68T07, 41A30; Secondary: 46E30, 28A05.}
\end{abstract}

\section{Introduction and main results}\label{sect-2}

This paper contributes to the ongoing dialogue between classical approximation theory and neural network theory by providing a more rigorous mathematical foundation for understanding how neural networks approximate $L_2$ functions. The result aligns with Wiener's classical framework while introducing new implications for modern machine learning architectures.

One of the foundational results regarding    approximation theory for neural networks was the {\em Universal Approximation Theorem} of \cite{cybenko} and \cite{hornik}, which establishes that 
a feedforward neural network with a single hidden layer and a continuous, non-polynomial activation function can approximate any continuous function on a compact subset of $\real^n$, given sufficiently many neurons. Subsequent extensions of this result addressed the analogous problem for Lebesgue spaces with a finite exponent \cite{hornik-2} and locally integrable spaces \cite{ParkSandberg}. 
Building upon these approximation theories, Wiener's Tauberian Theorems provide another vital tool in understanding function approximation since they 
offer criteria for when a set of translates of a function can densely span a function space like $L_1(\real)$  or $L_2(\real)$. Specifically, Wiener's Tauberian Theorem states that for a function $f\in L_1(\real)$, the translates of $f$ span $L_1(\real)$ if and only if the Fourier transform of $f$ has no zeros. 
As was already pointed out in \cite{cybenko}, this implies that  the set of all one-hidden layer neural networks of the form 
 $
    \sum_{i=1}^N \lambda_i f(x-\theta_i)  \text{ with }  N \in \mathbb{N},~\lambda_i,\theta_i \in \R, 
$
is dense in $L_1(\real)$ if  the activation function   $f\in L_1(\real)$ has   Fourier transform $\widehat{f}(x)\neq 0$. 
This result can be extended  to $L_2(\real)$, if one assumes that the Fourier transform of $f$ has only zeros on a set of Lebesgue measure zero -- making use of  Wiener's Tauberian Theorem for $L_2(\real)$.  
As one observes,  all previous (universal) approximation results in $L_p$-spaces have in common that the choice of suitable activations functions is somehow restricted. Let us mention in this context that a very nice overview of the state of the art concerning universal approximation and suitable activation functions can be found in \cite[Table~3.1]{jesus}. 
The main result of this paper can be considered as a generalization of Wiener's Tauberian Theorem in $L_2(\real)$. We  show that one can drop the assumption on the Fourier-transform  completely and use any $f\in L_2(\real)\setminus \{0\}$   in order to approximate any function in $L_2(\real)$: 
 \begin{thm}[{\bf Neural network approximation in $L_2(\R)$}]\label{thm:main_result_NN}
The set   of all one-hidden layer neural networks with activation function
  \(f \in L_2(\R)\) with  \(f \neq 0\),  i.e., 
    \[
    \sum_{i=1}^N \lambda_i f(\alpha_i x-\theta_i) \quad \text{ with }\quad  N \in \mathbb{N},~\lambda_i \in \mathbb{R},~\theta_i \in \mathbb{R},~\alpha_i \in \R\setminus \{0\}, 
    \]
    is  dense in \(L_2(\R)\). 
\end{thm}
We will deduce the above theorem from a more general statement about affine-invariant subspaces of $L_2(\real^n)$. We say a subspace \(U \subseteq L_2(\R^n)\) is \emph{affine-invariant} if \(f \in U\) implies \(f \circ A \in U\) for any invertible affine map \(A: \R^n \to \R^n\) (i.e., $A$ is a function of the form \(x \mapsto Ax:=Mx+v\) for some invertible $n\times n$ matrix \(M\) ($\det M\neq 0$)  and vector \(v\in \mathbb{R}^n\)). In particular, we are able to show: 

\begin{thm}[{\bf Affine-invariant subspaces in $L_2(\real^n)$}] \label{thm:main_result}
    Let \(U \subseteq L_2(\R^n)\) be a closed affine-invariant subspace. Then either \(U = L_2(\R^n)\) or \(U = \{0\}\).
\end{thm}

We remark in this context that the similar problem of translation and dilation invariant subspaces 
has recently been investigated in \cite{aleksandrov}. 
In particular, the results from  \cite[Thm.~6.1]{aleksandrov} imply that when $n=1$, Theorem \ref{thm:main_result} (and therefore also Theorem \ref{thm:main_result_NN})  remain even true in    $L_p(\real)$ for arbitrary $p\in (1,\infty)$. However, our method of proof is completely different and more elementary compared to \cite{aleksandrov} and of interest on its own.

Finally, one observes that Theorem \ref{thm:main_result_NN} is only formulated in dimension $n = 1$, while Theorem \ref{thm:main_result} holds in general dimensions. This is because shallow neural networks as
considered in Theorem \ref{thm:main_result_NN} never belong to $L_2(\mathbb{R}^n)$ in dimension $n > 1$; see   \cite[Thm.~6.1]{vanNuland} in this context.

There are 
several other significant contributions in this field of neural networks and approximation theory which deserve to be mentioned. 
In the last years, many directions have been explored 
for both shallow and deep neural networks. For example, in \cite{voigtlaender-2}  negative results  were studied, in \cite{voigtlaender} the universal approximation theorem for complex-valued neural networks was investigated, whereas \cite{CapelOcariz} deals the problem of approximating any function in a variable Lebesgue space using one-hidden layer neural networks. 
For neural networks with ReLU activation function, there are various recent papers on similar topics such as approximation for Besov spaces \cite{suz19}, \cite{siegel}, regression \cite{SH17},  optimization problems \cite{GS09}, and estimates for the errors obtained in the approximation \cite{pet99}, \cite{yar17}. Furthermore, for more regular activation functions there are also several recent articles such as   \cite{barron}, \cite{LTY20}, \cite{OK19} and many others, which focus on deep neural networks.

\section{Affine-invariant subspaces in \(L_2(\R^n)\)}

Our goal in this section is to prove Theorem \ref{thm:main_result}, which states that the only closed subspaces of \(L_2(\R^n)\) that are affine-invariant are \(L_2(\R^n)\) itself and the trivial subspace \(\{0\}\). 
In what follows we reduce  Theorem \ref{thm:main_result} to a purely measure-theoretic statement (Theorem \ref{thm:measure_theory}) using the following generalization of Wiener's Tauberian theorem for \(L_2(\mathbb{R}^n)\). 
\begin{thm}\label{thm:wtt}
    Let \(U \subseteq L_2(\R^n)\) be a closed translation-invariant subspace. Then there exists \(E \subseteq \mathbb{R}^n\) measurable such that 
    \[
    U = \mathcal{U}(E) := \left\{f \in L_2(\R^n): \  \widehat{f}|_{E} = 0\right\}.
    \]
\end{thm}
\begin{proof}
    {See \cite{ludwig2015}.}
\end{proof}

Here \(U\) being translation-invariant means that \(f(x) \in U\) implies \(f(x - \theta) \in U\) for all $\theta\in \real^n$ and \(\widehat{f}\) denotes the Fourier transform of \(f\). Note that in particular \(U\) is translation-invariant if it is affine-invariant.

Let \(\mathcal{M}\) denote the set of (Lebesgue) measurable subsets of \(\mathbb{R}^n\) and \(\vol \colon \mathcal{M} \to [0, \infty]\) the $n$-dimensional Lebesgue measure. We now embark on the proof of Theorem \ref{thm:measure_theory}. The proof relies on the Lebesgue density theorem, which we briefly state for the reader's convenience. 

\begin{thm}[{\bf Lebesgue density theorem}] \label{thm:lebesgue_density}
    Let \(E \in \mathcal{M}\). Then for almost all \(x \in E\)
    \[
    \lim_{r \searrow 0} \frac{\mathcal{L}^n(B_r(x) \cap E)}{\vol(B_r(x))} = 1, 
    \]
    where $B_r(x)$ denotes the open ball in $\real^n$ around $x$ with radius $r>0$.  
\end{thm}
\begin{proof}
    See \cite[Theorem 7.7]{rudin_analysis}.
\end{proof}

To deduce Theorem \ref{thm:measure_theory} from this we first need two lemmas. 

\begin{lem} \label{lemma:intersection_volume_transitive}
    Let \(k \geq 3\) and \(M_1, ..., M_k \in \mathcal{M}\). If \(M_2, ..., M_{k-1}\) have finite measure then 
    \[
    \vol(M_1 \cap M_k) \geq \sum_{i=1}^{k-1} \vol(M_i \cap M_{i+1}) - \sum_{i=2}^{k-1} \vol(M_i).
    \]
\end{lem}
\begin{proof}
   The  proof follows by induction on \(k\):  To start, let  \(k = 3\).   Then we have
    \begin{align*}
        \vol(M_1 \cap M_2) &= \vol(M_1 \cap M_2 \cap M_3) + \vol(M_1 \cap M_2 \cap M_3^\complement), \\
        \vol(M_2 \cap M_3) &= \vol(M_1 \cap M_2 \cap M_3) + \vol(M_1^\complement \cap M_2 \cap M_3), \\
        \vol(M_2) &\geq \vol(M_1 \cap M_2 \cap M_3) + \vol(M_1^\complement \cap M_2 \cap M_3) + \vol(M_1 \cap M_2 \cap M_3^\complement), 
    \end{align*}
    and because \(M_1 \cap M_2 \cap M_3\) has finite measure we see that 
    \begin{align*}
           \vol(M_2) \geq \vol(M_1 \cap M_2) + \vol(M_2 \cap M_3) - \vol(M_1 \cap M_2 \cap M_3).
    \end{align*}
    Therefore,  because \(M_2\) has finite measure, we obtain 
    \[
    \vol(M_1 \cap M_3) \geq \vol(M_1 \cap M_2 \cap M_3) \geq \vol(M_1 \cap M_2) + \vol(M_2 \cap M_3) - \vol(M_2), 
    \]
which shows the  claim for $k=3$. \\
    Now let \(k > 3\) and suppose the lemma is true for all \(m < k\). Let \(M_1, ..., M_k\) be measurable and let \(M_2, ..., M_{k-1}\) have finite measure. Using the inductive hypothesis on \(M_1, ..., M_{k-1}\) and \(M_1, M_{k-1}, M_k\) we obtain
    \begin{align*}
        \vol(M_1 \cap M_{k-1}) &\geq \sum_{i=1}^{k-2} \vol(M_i \cap M_{i+1}) - \sum_{i=2}^{k-2} \vol(M_i), \\
        \vol(M_1 \cap M_k) &\geq \vol(M_1 \cap M_{k-1}) + \vol(M_{k-1} \cap M_k) - \vol(M_{k-1}). 
    \end{align*}
    This yields
    \begin{align*}
            \vol(M_1 \cap M_k) &\geq \vol(M_1 \cap M_{k-1}) + \vol(M_{k-1} \cap M_k) - \vol(M_{k-1}) \\
            &\geq \sum_{i=1}^{k-2} \vol(M_i \cap M_{i+1}) - \sum_{i=2}^{k-2} \vol(M_i) + \vol(M_{k-1} \cap M_k) - \vol(M_{k-1})\\
            &= \sum_{i=1}^{k-1} \vol(M_i \cap M_{i+1}) - \sum_{i=2}^{k-1} \vol(M_i), 
    \end{align*}
which completes the proof. 
\end{proof}

For a field \(K\) let \(\GL(K,n)\) denote the set of invertible \(n\times n\) matrices with entries in \(K\). Moreover let \(\norm{\cdot}_{op}\) denote the operator norm induced by the 2-norm \(\norm{\cdot}_2\).
\begin{lem} \label{lem:ugly}
    Let \(\epsilon > 0\), \(a,b \in \R^n \setminus \{0\}\) and \(r_a, r_b > 0\) such that \(\frac{\norm{b}_2}{\norm{a}_2} = \frac{{r_b}}{{r_a}}\). Then there exists \(M \in \GL(\mathbb{Q},n)\) such that
    \begin{enumerate}[(i)]
        \item \(\left|\vol(M(B_{r_a}(a))) - \vol(B_{r_b}(b))\right| < \epsilon \cdot \vol(B_{r_b}(b))\),
        \item \(\left|\vol(M(B_{r_a}(a)) \cap B_{r_b}(b)) - \vol(B_{r_b}(b))\right| < 2\epsilon \cdot \vol(B_{r_b}(b))\).
    \end{enumerate}
\end{lem}
\begin{proof}
We start by showing there exists a matrix \(N \in \GL(\R, n)\) such that \(N(B_{r_a}(a)) = B_{r_b}(b)\): Let \(a_1 = \frac{a}{\norm{a}_2}\) and \(b_1 = \frac{b}{\norm{b}_2}\). Then using the Gram-Schmidt process we can find orthonormal bases \((a_1, ..., a_n)\) and \((b_1, ..., b_n)\) containing \(a_1\) and \(b_1\) respectively. Now let \(A\) be the orthogonal matrix with the \(a_i\) as columns and let \(B\) be the orthogonal matrix with the \(b_i\) as columns. Then \(BA^Ta_1 = b_1\) and thus \(\frac{\norm{b}_2}{\norm{a}_2}BA^Ta = b\). Hence putting \(N := \frac{\norm{b}_2}{\norm{a}_2}BA^T\) we found an invertible matrix that maps \(a\) to \(b\). Moreover,  the condition \(\frac{\norm{b}_2}{\norm{a}_2} = \frac{{r_b}}{{r_a}}\) together with the fact that \(N\) only rotates and scales and thus maps spheres to spheres   ensures that \(N(B_{r_a}(a)) = B_{r_b}(b)\). \\
   \noindent Next we can choose a matrix that is sufficiently close to \(N\) as our matrix \(M\): 
  Because the determinant is a continuous map there exists \(\delta_1 > 0\) such that 
    \begin{equation} \label{eq:det_continuous}
        \norm{M-N}_{op} < \delta_1 \implies \abs{\det(M) - \det(N)} < \frac{\epsilon \cdot \vol(B_{r_b}(b))}{\vol(B_{r_a}(a))}
    \end{equation}
    for all \(M \in \GL(\R,n)\). Furthermore there exists \(\delta_2 > 0\) such that 
    \begin{equation} \label{eq:ballborder}
    \frac{\vol(B_{r_b + \delta_2}(b) \setminus B_{r_b}(b))}{\vol(B_{r_b}(b))} = \frac{(r_b + \delta_2)^n - r_b^n}{r_b^n} < \epsilon.
    \end{equation}
    Finally choose \(R>0\) such that \(B_{r_a}(a) \subseteq B_R(0)\) and define \(\delta := \min\{\delta_1, \frac{\delta_2}{R}\}\). Then, because \(\mathbb{Q}\) is dense in \(\mathbb{R}\) and because the set of invertible matrices is open, there exists \(M \in \GL(\Q, n)\) such that \(\norm{M-N}_{op} < \delta\). Now we have 
    \begin{equation*} 
    \abs{\vol(M(B_{r_a}(a))) - \vol(B_{r_b}(b))} = \big|\abs{\det(M)} - \abs{\det(N)}\big| \cdot \vol(B_{r_a}(a)) \leq \epsilon \cdot \vol(B_{r_b}(b))
    \end{equation*}
    by (\ref{eq:det_continuous}), which shows that \(M\) satisfies (i).\\
    Next let \(x \in B_{r_a}(a)\). Then
    \begin{equation*}
        \norm{Mx-b}_2 \leq \norm{Mx-Nx}_2 + \norm{Nx-b}_2 \leq \norm{M-N}_{op} \cdot \norm{x}_2 + r_b \leq \delta_2 \cdot \frac{\norm{x}_2}{R} + r_b \leq \delta_2 + r_b,
    \end{equation*}
    which shows that \(M(B_{r_a}(a)) \subseteq B_{r_b + \delta_2}(b).\) \\
    Thus \(M(B_{r_a}(a))\) is the disjoint union of \(M(B_{r_a}(a)) \cap B_{r_b}(b)\) and \(M(B_{r_a}(a)) \cap B_{r_b + \delta_2}(b) \setminus B_{r_b}(b)\) and therefore using (\ref{eq:ballborder}) we have
    \begin{align*}
        \vol(M(B_{r_a}(a)) \cap B_{r_b}(b)) &= \vol(M(B_{r_a}(a))) - \vol(M(B_{r_a}(a)) \cap B_{r_b + \delta_2}(b) \setminus B_{r_b}(b)) \\
        &\geq \vol(M(B_{r_a}(a))) - \vol(B_{r_b + \delta_2}(b) \setminus B_{r_b}(b)) \\
        &\geq \vol(M(B_{r_a}(a))) - \epsilon \cdot \vol(B_{r_b}(b)).
    \end{align*}
    Finally using the fact that \(M\) satisfies (i) we obtain
    \begin{align*}
        \left|\vol(M(B_{r_a}(a)) \cap B_{r_b}(b)) - \vol(B_{r_b}(b))\right| &= \vol(B_{r_b}(b)) - \vol(M(B_{r_a}(a)) \cap B_{r_b}(b)) \\
        &\leq \vol(B_{r_b}(b)) - \vol(M(B_{r_a}(a))) + \epsilon \cdot \vol(B_{r_b}(b)) \\
        &\leq \abs{\vol(B_{r_b}(b)) - \vol(M(B_{r_a}(a)))} + \epsilon \cdot \vol(B_{r_b}(b)) \\
        &\leq 2\epsilon\vol(B_{r_b}(b)),
    \end{align*}
    showing that \(M\) also satisfies (ii).
\end{proof}

We are now ready to prove Theorem \ref{thm:measure_theory}. For \(M,N \in \mathcal{M}\) we write \(M \aeq N\) if \(M\) and \(N\) only differ on a set of measure zero, that is if \(\vol((M \setminus N) \cup (N \setminus M))=0\). Moreover for \(A \subseteq \R^n\) we define \(\GL(\mathbb{Q},n) \cdot A := \{Ma : M \in \GL(\Q,n), a \in A \}\).

\begin{thm} \label{thm:measure_theory}
    Let \(A \in \mathcal{M}\) with positive measure. Then 
    \[
    \GL(\mathbb{Q},n) \cdot A \aeq \mathbb{R}^n.
    \]
\end{thm}
\begin{proof}
    Let \(\epsilon := 0.1\) and \(B := (\GL(\mathbb{Q},n) \cdot A)^\complement\). Suppose \(B\) had positive measure. Then by the Lebesgue density theorem there would exist (nonzero) points of density \(a \in A, b \in B\) for \(A\) and \(B\). Hence there would exist \(R > 0\) such that 
    \begin{align*}
        \vol(A \cap B_r(a)) &\geq (1 - \epsilon) \cdot \vol(B_r(a)), \\
        \vol(B \cap B_r(b)) &\geq (1 - \epsilon) \cdot \vol(B_r(b))
    \end{align*}
    for all \(r < R\). Now choose real numbers \(r_a\) and \(r_b\) satisfying \(0 < r_a, r_b < R\) and \(r_b \cdot \norm{a}_2 = r_a \cdot \norm{b}_2\). Then by Lemma \ref{lem:ugly} there exists \(M \in \GL(\mathbb{Q},n)\) such that
    \begin{align*}
        \abs{\vol(M(B_{r_a}(a))) - \vol(B_{r_b}(b))} &\leq \epsilon \cdot \vol(B_{r_b}(b)),\\
        \left|\vol(M(B_{r_a}(a)) \cap B_{r_b}(b)) - \vol(B_{r_b}(b))\right| &\leq 2\epsilon \cdot \vol(B_{r_b}(b)).
    \end{align*}
    In particular, we have
    \begin{align*}
        \vol(M(B_{r_a}(a)) \cap B_{r_b}(b)) &\geq (1 - 2\epsilon)\vol(B_{r_b}(b)), \\
        \vol(M(B_{r_a}(a))) &\leq (1 + \epsilon)\vol(B_{r_b}(b)).
    \end{align*}
    Thus using Lemma \ref{lemma:intersection_volume_transitive} we obtain
    \begin{align*}
        \vol(B \cap M(A)) &\geq \vol(B \cap B_{r_b}(b)) + \vol(B_{r_b}(b) \cap M(B_{r_a}(a))) + \vol(M(B_{r_a}(a)) \cap M(A)) \\
        & ~~~ - \vol(B_{r_b}(b)) - \vol(M(B_{r_a}(a))) \\
        &\geq (1-\epsilon)\vol(B_{r_b}(b)) + (1 - 2\epsilon)\vol(B_{r_b}(b)) + \det(M) (1 - \epsilon) \vol(B_{r_a}(a)) \\
        & ~~~ - \vol(B_{r_b}(b)) - \det(M)\vol(B_{r_a}(a)) \\
        &= (1 - 3\epsilon)\vol(B_{r_b}(b)) - \epsilon\cdot\vol(M(B_{r_a}(a)))  \\
        &\geq (1-3\epsilon - \epsilon(1 + \epsilon))\vol(B_{r_b}(b))  = (1-4\epsilon - \epsilon^2)\vol(B_{r_b}(b))  > 0
    \end{align*}
    contradicting the definition of \(B\) as the complement of \(\GL(\mathbb{Q},n) \cdot A\).
\end{proof}

Having established Theorem \ref{thm:measure_theory} we only need one more small lemma in order to be able to proof Theorem \ref{thm:main_result}. In what follows we use the following convention: We say that \(f \in L_2(\real^n)\) has no zeros in \(E \in \mathcal{M}\) if for some (and thus every) representative \(\varphi_f\) of \(f\) it holds that 
$
\vol(\varphi_f^{-1}(\{0\}) \cap E) = 0.
$
\begin{lem}\label{lemma:crisp} 
    Let \(E \in \mathcal{M}\), \(f \in \mathcal{U}(E)=\left\{g \in L_2(\R^n): \  \widehat{g}|_{E} = 0\right\}\), and \(M \in \mathcal{M}\) such that \(\widehat{f}\) has no zeros in \(M\). Then \(\vol(M\cap E) = 0\).
\end{lem}
\begin{proof}
    Because \(f \in \mathcal{U}(E)\) we have \(\widehat{f}|_{M \cap E} = 0\). If \(M \cap E\) had nonzero measure this would contradict \(\widehat{f}\) having no zeros in \(M\).
\end{proof}

With all this preparation, we can now finally characterize  affine-invariant subspaces in $L_2(\real^n)$. 

\begin{proof}{\em (of Theorem \ref{thm:main_result})}: 
    Suppose \(U \neq \{0\}\). Then there exists some \(f \in U\) with \(f \neq 0\) and thus \(\widehat{f} \neq 0\) by Plancherel's Theorem (cf. \cite[Theorem 1.6.1]{rudin_fourier_analysis}) and there exists \(A \in \mathcal{M}\) with positive measure such that \(\widehat{f}\) has no zeros in \(A\) (take for example the complement of the zero set of some representative of \(\widehat{f}\)). Then by Theorem \ref{thm:measure_theory} we have \(\GL(\mathbb{Q},n)\cdot A \aeq \mathbb{R}^n\). Moreover,  for any \(M \in \GL(\Q,n)\) the Fourier transform of \(f \circ M\) has no zeros in \(M^T(A)\). This is because we have
    \[
        \widehat{f \circ M}(x) = \frac{1}{|\det(M)|} \cdot \widehat{f}(M^{-T}x) \quad  \text{ a.e.},
    \]
    which can be verified directly using the definition of the Fourier transform.

    Thus as \(U\) is affine-invariant we have \(f \circ M \in U\) and Lemma \ref{lemma:crisp} yields \(\vol(M^T(A) \cap E) = 0\) for any \(M \in \GL(\mathbb{Q},n)\), where \(E \in \mathcal{M}\) is such that \(\mathcal{U}(E) = U\) (see Theorem \ref{thm:wtt}). Then
    \[
    \vol(E) = \vol(E \cap (\GL(\mathbb{Q},n)\cdot A)) = \vol\left(\bigcup_{M \in \GL(\mathbb{Q},n)} (E \cap M^T(A))\right) \leq \sum_{M \in \GL(\mathbb{Q},n)} \vol(E \cap M^T(A)) = 0,
    \]
    which implies
    \begin{align*}
        U = \mathcal{U}(E) = \{g \in L_2(\real^n) :\  \widehat{g}|_E = 0\} = L_2(\real^n), 
    \end{align*}
    because all functions are equivalent to the zero function when restricted to a set of measure zero.
\end{proof}

\section{Application to neural networks}

We now use the result from Theorem \ref{thm:main_result} when $n=1$ and investigate when one-hidden layer neural networks are dense in \(L_2(\R)\).

\begin{proof}{\em (of Theorem \ref{thm:main_result_NN})}: 
  Note that the neural networks considered in Theorem \ref{thm:main_result_NN} form an affine-invariant subspace of \(L_2(\R)\). Thus, if for $f\in L_2(\real)\setminus \{0\}$ we let   \[N_f :=\left\{g\in L_2(\real): \ 
 g(x)=\sum_{i=1}^N \lambda_i f(\alpha_i x-\theta_i), \    N \in \mathbb{N},~\lambda_i, \theta_i \in \mathbb{R},~\alpha_i \in \R\setminus \{0\} 
\right\}\subseteq L_2(\R)
\] 
denote the set of these neural networks, it follows that  \(\overline{N_f}\) is a closed affine-invariant subspace and Theorem \ref{thm:main_result} implies  \(\overline{N_f} = L_2(\R)\) for nonzero  \(f \in L_2(\R)\).  
\end{proof}

\begin{rem}
 Most activation functions used in practice are not in \(L_2(\R)\) and thus Theorem \ref{thm:main_result_NN} does not apply to them directly. However, by Theorem \ref{thm:main_result} it suffices if there is some nonzero \(L_2\) function that can be approximated (using a one-hidden layer neural network with such an activation function) to deduce all \(L_2\) functions can be. 
Consider for example the ReLU activation function \(\mathrm{ReLU}(x) = \max\{0,x\}\). Then the hat function can be represented as 
\begin{align*}
    (\relu(x)-\relu(x+1))-(\relu(x+1)-\relu(x+2)) = \left.\begin{cases}
        0, & \text{ \(x \leq -2\)} \\
        x+2, & \text{ \(-2 \leq x \leq -1\)} \\
        -x, & \text{ \(-1 \leq x \leq 0\)} \\
        0, & \text{ \(0 \leq x\)}
    \end{cases}\right\}  \in L_2(\R).
\end{align*}
Hence there exists a nonzero function in \(N_{\relu}\) belonging to $L_2(\real)$ and from the structure of $N_{\relu}$ and Theorem \ref{thm:main_result} we deduce  that \(\overline{N_{\relu}} = L_2(\R)\).

\end{rem}

\medskip 


{\small 

\bibliographystyle{plain}
\bibliography{refs.bib}

}

\end{document}